\documentclass{article}
\usepackage[a4paper, total={6in, 8in}]{geometry}
\usepackage{amsmath,amsthm,amssymb,mathrsfs,amstext, titlesec,enumitem,comment}
\usepackage{units}
\usepackage{amstext}
\usepackage{mathdots}
\usepackage{xcolor}
\usepackage{hyperref}
\makeatletter

\titleformat{\section}
{\normalfont\Large\bfseries}{\thesection}{1em}{}
\titleformat*{\subsection}{\Large\bfseries}
\titleformat*{\subsubsection}{\large\bfseries}
\titleformat*{\paragraph}{\large\bfseries}
\titleformat*{\subparagraph}{\large\bfseries}

\renewcommand{\@seccntformat}[1]{\csname the#1\endcsname. }

\renewenvironment{abstract}{%
    \if@twocolumn
      \section*{\abstractname}%
    \else 
      \begin{center}%
        {\bfseries \Large\abstractname\vspace{\z@}}
      \end{center}%
      \quotation
    \fi}
    {\if@twocolumn\else\endquotation\fi}
    
\makeatother
\theoremstyle{plain}
\newtheorem{thm}{Theorem}[section]

\newtheorem{lem}[thm]{Lemma}
\theoremstyle{definition}
\newtheorem{defn}[thm]{Definition}

\theoremstyle{numberfirst}

\newcommand{\N}{\mathbb{N}}

\newcommand{\thistheoremname}{}
\newtheorem*{genericthm*}{\thistheoremname}
\newenvironment{namedthm*}[1]
  {\renewcommand{\thistheoremname}{#1}%
   \begin{genericthm*}}
  {\end{genericthm*}}

\date{\vspace{-5ex}}

\begin{document}

\title{\textbf{Combined exponential patterns in multiplicative $IP^{\star}$ sets}}

\author{ 
Pintu Debnath\\  \textit{pintumath1989@gmail.com}\footnote{ Department of Mathematics, Basirhat College, Basirhat
-743412, North 24th parganas, West Bengal, India.}\\
\and	
	Sayan Goswami\\  \textit{sayangoswami@imsc.res.in}\footnote{The Institute of Mathematical Sciences, A CI of Homi Bhabha National Institute, CIT Campus, Taramani, Chennai 600113, India.}
	}

\maketitle

\begin{abstract}
$IP$ sets play fundamental role in arithmetic Ramsey theory. A set is called an additive $IP$ set  if it is of the form $FS\left(\langle x_{n}\rangle_{n\in \mathbb{N}}\right)=\left\{ \sum_{t\in H}x_{t}:H\right.$ is a nonempty finite subset of $\left.\mathbb{N}\right\}$, whereas it is called a multiplicative $IP$
set if it is of the form
$FP\left(\langle x_{n}\rangle_{n\in \N}\right)=\left\{ \prod_{t\in H}x_{t}:H\right.$ is a nonempty finite subset of $\left. \mathbb{N}\right\}$
for some injective sequence $\langle x_{n}\rangle_{n\in \mathbb{N}}.$ An additive
$IP^{\star}$ (resp. multiplicative $IP^{\star}$) set is a set which
intersects every additive $IP$ set (resp. multiplicative $IP$ set).
In \cite{key-1}, V. Bergelson and N. Hindman studied how rich additive
$IP^{\star}$ sets are. They proved additive $IP^{\star}$ sets ($AIP^{\star}$
in short) contain finite sums and finite products of a single sequence.
An analogous study was made by A. Sisto in\cite{key-3}, where he proved that  multiplicative $IP^{\star}$
sets ($MIP^{\star}$ in short)  contain exponential tower\footnote{will be defined later}
and finite product of a single sequence. However exponential patterns can be defined in two different ways. In this article we will prove
that $MIP^{\star}$ sets contain two different exponential
patterns and finite product of a single sequence. This immediately
improves the result of A. Sisto. We also construct a $MIP^\star$ set, not arising from the recurrence of measurable dynamical systems. Throughout our work we will use the machinery
of the algebra of the Stone-\v{C}ech Compactification of $\mathbb{N}$. 
\end{abstract}

\section{Introduction}

The origin of $IP$ sets dates back to Hindman's work \cite{key-4},
where he proved that for any finite coloring of the set of positive
integers $\mathbb{N}$, there exists a monochromatic copy of an additive
$IP$ set. Here ``coloring'' means disjoint partition, and a pattern
``monochromatic'' means if it is included in one piece of the partition.
Passing to the map $n\rightarrow2^{n}$ for each $n\in\mathbb{N}$,
we immediately have a monochromatic copy of a multiplicative $IP$
set. Let $\beta\mathbb{N}$ be the set of all ultrafilters\footnote{For details on the algebra of ultrafilters we refer the book \cite{key-2}
of N. Hindman and D. Strauss.} over $\mathbb{N}$, and $E\left(\beta\mathbb{N},+\right)$ (resp.
$E\left(\beta\mathbb{N},\cdot\right)$) be the collection of all idempotents
in $\left(\beta\mathbb{N},+\right)$ (resp. $\left(\beta\mathbb{N},\cdot\right)$).
One can show that a set $A$ is additive $IP$ (resp. multiplicative
$IP$ set) if there exists $p\in E\left(\beta\mathbb{N},+\right)$
(resp. $p\in E\left(\beta\mathbb{N},\cdot\right)$) such that $A\in p.$
Hence a set $A$ is $AIP^{\star}$ (resp. $MIP^{\star}$) if and only
if $A\in p$ for all $p\in E\left(\beta\mathbb{N},+\right)$ (resp.
$p\in E\left(\beta\mathbb{N},\cdot\right)$). Define $\mathcal{P}_f(\N)$ to be the set of all nonempty finite subsets of $\N$. For any $IP$ set $FS\left(\langle x_{n}\rangle_{n}\right)$,
a sum subsystem of $FS\left(\langle x_{n}\rangle_{n}\right)$ is of
the form $FS\left(\langle y_{n}\rangle_{n}\right)$, where for each
$n\in\mathbb{N},$ $y_{n}$ is defined as follows.
\begin{itemize}
\item There exists a sequence $\langle H_{n}\rangle_{n}$ in $\mathcal{P}_{f}(\mathbb{N})$
satisfying $H_{i}\cap H_{j}=\emptyset$ for each $i\neq j$, and 
\item $y_{n}=\sum_{t\in H_{n}}x_{t}.$ 
\end{itemize}
In \cite{key-1}, V. Bergelson and N. Hindman proved the following
result, which addresses that any $AIP^{\star}$ set contains combined
additive and multiplicative patterns. 
\begin{thm}
\label{aip}Let $A$ be an $AIP^{\star}$ set, and $\langle x_{n}\rangle_{n\in\mathbb{N}}$
be any sequence. Then there exists a sum subsystem $FS\left(\langle y_{n}\rangle_{n\in\mathbb{N}}\right)$
of $FS\left(\langle x_{n}\rangle_{n\in\mathbb{N}}\right)$ such that
$FS\left(\langle y_{n}\rangle_{n\in\mathbb{N}}\right)\cup FP\left(\langle y_{n}\rangle_{n\in\mathbb{N}}\right)\subset A.$ 
\end{thm}

An immediate question appears what about $MIP^{\star}$ sets? In \cite{key-3},
A. Sisto was able to show that these sets contain combined multiplicative
and exponential patterns.  To state his theorem  explicitly, we need the
following definitions. 
\begin{defn}
For any sequences $\langle x_{n}\rangle_{n=1}^{\infty},$ define 
\end{defn}

\begin{enumerate}
\item for any $N\in\mathbb{N}$, $EXP_{1}\left(\langle x_{n}\rangle_{n=1}^{N}\right)=\left\{ x_{i_{k}}^{x_{i_{k-1}}^{\iddots^{x_{i_{1}}}}}:1\leq i_{1}<i_{2}<\cdots<i_{k}\leq N\right\} ,$
\item for any $N\in\mathbb{N}$, $EXP_{2}\left(\langle x_{n}\rangle_{n=1}^{N}\right)=\left\{ x_{i_{1}}^{x_{i_{2}}\cdots x_{i_{k}}}:1\leq i_{1}<i_{2}<\cdots<i_{k}\leq N\right\} ,$
\item $EXP_{1}\left(\langle x_{n}\rangle_{n=1}^{\infty}\right)=\left\{ x_{i_{k}}^{x_{i_{k-1}}^{\iddots^{x_{i_{1}}}}}:1\leq i_{1}<i_{2}<\cdots<i_{k},k\in\mathbb{N}\right\} ,$
\item $EXP_{2}\left(\langle x_{n}\rangle_{n=1}^{\infty}\right)=\left\{ x_{i_{1}}^{x_{i_{2}}\cdots x_{i_{k}}}:1\leq i_{1}<i_{2}<\cdots<i_{k},k\in\mathbb{N}\right\} .$
\end{enumerate}
The following Corollary of Sisto's addresses exponential properties
of $MIP^{\star}$ sets.
\begin{thm}
\label{mip} \textup{\cite[Corollary 16]{key-3}} Let $A$ be a $MIP^{\star}$set.
Then there exist sequences $\langle x_{n}\rangle_{n=1}^{\infty}$
and $\langle y_{n}\rangle_{n=1}^{\infty}$ such that
\begin{enumerate}
\item $FP\left(\langle x_{n}\rangle_{n=1}^{\infty}\right)\bigcup EXP_{1}\left(\langle x_{n}\rangle_{n=1}^{\infty}\right)\subseteq A,$
and 
\item $FS\left(\langle y_{n}\rangle_{n=1}^{\infty}\right)\bigcup EXP_{2}\left(\langle y_{n}\rangle_{n=1}^{\infty}\right)\subseteq A.$
\end{enumerate}
\end{thm}

A natural question appears whether it is possible to provide a joint
extension of both (1) and (2) in Theorem \ref{mip}. That means, for
each $n\in\mathbb{N},$ can we choose $x_{n}=y_{n}$ in Theorem \ref{mip}. 

In this article we  provide a partial answer to this question
by proving the following theorem.
\begin{thm}
\label{mipnew} Let $A$ be a $MIP^{\star}$
set. Then there exist sequences $\langle x_{n}\rangle_{n=1}^{\infty}$
such that $$FP\left(\langle x_{n}\rangle_{n=1}^{\infty}\right)\bigcup EXP_{1}\left(\langle x_{n}\rangle_{n=1}^{\infty}\right)\cup EXP_{2}\left(\langle x_{n}\rangle_{n=1}^{\infty}\right)\subseteq A.$$
\end{thm}

\section*{Acknowledgement:} We are thankful to Sourav Kanti Patra for discussions in several occasions.

\section{Our rseults}

\subsection{Proof of Theorem \ref{mipnew}}
Ellis theorem \cite[Theorem 2.5]{key-2} tells us about the existence of idempotents in topological
semigroups. It says that every compact Hausdorff right topological
semigroup contains idempotents. It is a routine exercise to prove
that $cl\left(E\left(\beta\mathbb{N},+\right)\right)$ is a left ideal
of $\left(\beta\mathbb{N},\cdot\right).$ As left ideals contain minimal left ideals and these are closed,
we can apply Ellis theorem to conclude that $cl\left(E\left(\beta\mathbb{N},+\right)\right)\bigcap E\left(\beta\mathbb{N},\cdot\right)\neq\emptyset.$
To prove Theorem \ref{mipnew}, we will rely on the elements of $cl\left(E\left(\beta\mathbb{N},+\right)\right)\bigcap E\left(\beta\mathbb{N},\cdot\right).$
\begin{proof}[\textbf{Proof of Theorem \ref{mipnew}:}]
 Let $p\in cl\left(E\left(\beta\mathbb{N},+\right)\right)\bigcap E\left(\beta\mathbb{N},\cdot\right)$,
and $A$ be a $MIP^{\star}$set. As $A\in p$, and $p=p\cdot p,$ denote by $A^{\star}=\left\{ x\in A:x^{-1}A\in p\right\} \in p$.
Choose $x_{1}\in A^{\star}$. As $A$ is a $MIP^{\star}$ set, we have $B_{1}=\left\{ n:n^{x_{1}}\in A\right\} $
is a $MIP^{\star}$ set. Also by \cite[Lemma 13]{key-3}, the set $C_{1}=\left\{ m:x_{1}^{m}\in A\right\} $
is $AIP^{\star}$ set. Set $$D_{1}=B_{1}\cap A^{\star}\cap x_{1}^{-1}A^{\star}\in p.$$

As $p\in cl\left(E\left(\beta\mathbb{N},+\right)\right)$, we have
$C_{1}\cap D_{1}\neq\emptyset$. Let $x_{2}\in C_{1}\cap D_{1}$.
Then $x_{2}\in B_{1}\text{ and this implies }x_{2}^{x_{1}}\in A$. As $x_{2}\in C_{1}$ and so $x_{1}^{x_{2}}\in A$.
Again $x_{2}\in A^{\star}\cap x_{1}^{-1}A^{\star}$, this implies
$\left\{ x_{1},x_{2},x_{1}x_{2}\right\} \subset A.$ 
Hence $\{x_1^{x_2},x_2^{x_1}\}\subset A$, and $\{x_1,x_2,x_1x_2\}\subset A^\star$.

Inductively assume
that for some $N\in\mathbb{N}$, we have $x_{1},x_{2},\ldots,x_{N}$
such that
\begin{enumerate}
\item $EXP_{1}\left(\langle x_{n}\rangle_{n=1}^{N}\right)\bigcup EXP_{2}\left(\langle x_{n}\rangle_{n=1}^{N}\right)\subset A$
and 
\item $FP\left(\langle x_{n}\rangle_{n=1}^{N}\right)\subset A^{\star}$. 
\end{enumerate}
For each $z_{1}\in EXP_{1}\left(\langle x_{n}\rangle_{n=1}^{N}\right)$,
let $B_{z_{1}}=\left\{ n:n^{z_{1}}\in A\right\} $ is a $MIP^{\star}$set.
For each $z_{2}\in EXP_{2}\left(\langle x_{n}\rangle_{n=1}^{N}\right)$,
let $C_{z_{2}}=\left\{ m:z_{2}^{m}\in A\right\} $ is an $AIP^{\star}$set.
Hence $\bigcap_{z_{2}\in EXP_{2}\left(\langle x_{n}\rangle_{n=1}^{N}\right)}C_{z_{2}}$
is an $AIP^{\star}$ set. So 
\[
D_{n+1}=\bigcap_{z_{1}\in EXP_{1}\left(\langle x_{n}\rangle_{n=1}^{N}\right)}B_{z_{1}}\cap A^{\star}\cap\bigcap_{y\in FP\left(\langle x_{n}\rangle_{n=1}^{N}\right)}y^{-1}A^{\star}\in p.
\]
Again $p\in cl\left(E\left(\beta\mathbb{N},+\right)\right)$, hence
$E_{N+1}=\bigcap_{z_{2}\in EXP_{2}\left(\langle x_{n}\rangle_{n=1}^{N}\right)}C_{z_{2}}\cap D_{N+1}\neq\emptyset$,
and let $x_{n+1}\in E_{N+1}$. Then, $x_{N+1}^{z}\in A$ for all $z\in EXP_{1}\left(\langle x_{i}\rangle_{i=1}^{N}\right)$,
and $y^{x_{N+1}}\in A$ for all $y\in EXP_{2}\left(\langle x_{i}\rangle_{i=1}^{N}\right)$.
Again $x_{N+1}\in A^{\star}\cap\bigcap_{y\in FP\left(\langle x_{n}\rangle_{n=1}^{N}\right)}y^{-1}A^{\star}$
implies $FP\left(\langle x_{i}\rangle_{i=1}^{N}\right)\subset A^{\star}$.

Hence we have
\begin{enumerate}
\item $EXP_{1}\left(\langle x_{n}\rangle_{n=1}^{N+1}\right)\bigcup EXP_{2}\left(\langle x_{n}\rangle_{n=1}^{N+1}\right)\subset A,$
and 
\item $FP\left(\langle x_{n}\rangle_{n=1}^{N+1}\right)\subset A^{\star}$. 
\end{enumerate}
This completes the induction.
\end{proof}

\subsection{$MIP^\star$ sets not coming from dynamics}
If $A$ is an $MIP^{\star}$ set, then for any prime $q,$ the set
$B=\left\{ n:q^{n}\in A\right\} $ is an $AIP^{\star}$ set. For any
$n\in\mathbb{N}$, and $X\subseteq\mathbb{N}$, define $n^{X}=\left\{ n^{x}:x\in X\right\} .$
Then we can write $A=q^{B}\cup C,$ where $q^{B}\cap C=\emptyset.$
If we choose $D$ to be any $AIP^{\star}$ set, then it will be easy
to prove that $q^{B\cap D}\cup C$ is also a $MIP^{\star}$ set.
\begin{lem}
\label{ess} Let $q$ be a prime number. Then for any given $MIP^{\star}$
set $B$, we can find a $MIP^{\star}$ set $A\subseteq B$ such that for all $n\in\mathbb{N}$,
$\nicefrac{A}{q^{n}}$is not $MIP^{\star}$ set.
\end{lem}

\begin{proof}
Given a $MIP^{\star}$ set $B$ and a prime $q,$ we have an $AIP^{\star}$
set $C$ such that $B=p^{C}\cup D.$ From \cite[Theorem 16.32]{key-2},
we know that there exists an $AIP^{\star}$ set $F$ such that for
all $n\in\mathbb{N}$, $-n+F$ is not an $AIP^{\star}$ set. Let $E=F\cap C,$
and consider $A=p^{E}\cup D$. As $E$ is an $AIP^{\star}$ set, we have $A$
is a $MIP^{\star}$ set.

We claim that for any $n\in\mathbb{N}$, $\nicefrac{A}{q^{n}}$ is
not $MIP^{\star}$ set. 

To prove our claim, if possible assume for some $n\in\mathbb{N}$,
$\nicefrac{A}{q^{n}}$ is a $MIP^{\star}$ set. Then $G=\left\{ m:q^{m}\in\nicefrac{A}{q^{n}}\right\} $
is $AIP^{\star}$ set. But by our choice $G=-n+E$ is not an $AIP^{\star}$
set. Hence $\nicefrac{A}{q^{n}}$ is not a $MIP^{\star}$
set.
\end{proof}
Certain $IP^{\star}$ sets arise from the recurrence of the measurable
dynamical systems. Such sets are known as dynamical $IP^{\star}$
sets. These sets are defined as follows.
\begin{defn}
\cite[Definition 19.34]{key-2} Let $(S,\cdot)$ be a semigroup, and
a subset $C$ of $S$ is called a dynamical $IP^{\star}$ set if and
only if there exists a measure preserving system $\left(X,\mathcal{B},\mu,\langle T_{s}\rangle_{s\in S}\right),$
and an $A\in\mathcal{B}$ with $\mu\left(A\right)>0$ such that $\text{\ensuremath{\left\{  s\in S:\mu\left(A\cap T_{s}^{-1}\left[A\right]\right)>0\right\} } \ensuremath{\subseteq C} }.$
\end{defn}

The following theorem says that infinitely many translations dynamical
$IP^{\star}$ sets are always dynamical $IP^{\star}.$
\begin{thm}
\label{cite} Let $B$ be a dynamical $MIP^{\star}$ set of $\left(\mathbb{N},\cdot\right)$.
Then there is a dynamical $MIP^{\star}$ set $C\subseteq B$ such
that for each $n\in C$, $\nicefrac{C}{n}$ is a dynamical $MIP^{\star}$
set (and hence $\nicefrac{B}{n}$ is a dynamical $MIP^{\star}$ set).
\end{thm}

\begin{proof}
Same as the proof of \cite[Theorem 19.35]{key-2}.
\end{proof}

Now we are in the position to construct a $MIP^\star$ set which is not dynamical.
\begin{thm}
There exists a $MIP^{\star}$ set in $\left(\mathbb{N},\cdot\right)$
which is not dynamical $MIP^{\star}$ set.
\end{thm}

\begin{proof}
Let $q$ be a prime and $A$ be a $MIP^{\star}$ set as in the proof
of Lemma \ref{ess}. If possible, assume that $A$ is also a dynamical
$MIP^{\star}$ set. Then from Theorem \ref{cite}, we have $F=\left\{ n:\nicefrac{A}{n}\right.$ is a dynamical $\left. MIP^{\star}\text{ set}\right\} $
is also a dynamical $MIP^{\star}$ set.

As $F$ is a dynamical $MIP^{\star}$ set, and $FP\left(\langle q^{n}\rangle_{n\in\mathbb{N}}\right)\cap F\neq\emptyset$,
we have some $k\in\mathbb{N}$ such that $q^{k}\in F.$ Hence $\nicefrac{F}{p^{k}}$
is $MIP^{\star}$ set. But by the construction of $A$ in Lemma
\ref{ess}, $\nicefrac{A}{q^{k}}$ is not a $MIP^{\star}$ set $A$.
Hence $A$ is a $MIP^{\star}$ set but not a dynamical $MIP^{\star}$set.

\end{proof}


\begin{thebibliography}{1}
\bibitem{key-1} V. Bergelson and N. Hindman. On $IP^{\star}$sets
and central sets, Combinatorica, 14(3), 1994, 269-277.

\bibitem{key-4} N. Hindman, Finite sums from sequences within cells of a partition of  $\N$, J. Comb. Theory (series A), 17 (1974), 1-11.

\bibitem{key-2} N. Hindman, D. Strauss, Algebra in the Stone-\v{C}ech
Compactification: Theory and Application, de gruyter, Berlin, 2012.

\bibitem{key-3} A. Sisto, Exponential triples, The Electronic journal
of combinatorics, 18 (2011), \#P147.
\end{thebibliography}
\end{document}